\documentclass[a4paper,12pt]{article}

\usepackage[T1]{fontenc}
\usepackage{enumerate, comment}
\usepackage{amsmath,amssymb,amsthm}
\usepackage{mathtools}
\usepackage{graphicx}
\usepackage{bbm}
\usepackage{dsfont}
\usepackage{cite, color}   

\definecolor{Red}{cmyk}{0,1,1,0}

\definecolor{verde}{cmyk}{1,0,1,0}

\definecolor{azul}{cmyk}{1,1,0,0}

\usepackage[hidelinks,bookmarksnumbered]{hyperref}
\hypersetup{pdfstartview=FitH}

\usepackage{tikz}
\usetikzlibrary{patterns}
\usetikzlibrary{calc}
\usetikzlibrary{intersections}
\usetikzlibrary{through}
\usetikzlibrary{backgrounds}

\def\ra{\rightarrow}

\def\Z{\mathbb{Z}}
\def\P{\mathbb{P}}
\def\E{\mathbb{E}}
\def\L{\Lambda}

  \newcounter{constant}

\newcommand{\I}{\mathds{1}}

\newcommand{\cC}{\ensuremath{\mathcal{C}}}
\newcommand{\cD}{\ensuremath{\mathcal{D}}}

\newcommand{\NN}{\ensuremath{\mathbb{N}}}
\newcommand{\PP}{\ensuremath{\mathbb{P}}}
\newcommand{\RR}{\ensuremath{\mathbb{R}}}
\newcommand{\ZZ}{\ensuremath{\mathbb{Z}}}

\newcommand{\C}{\ensuremath{\mathcal{C}}}
\newcommand{\D}{\ensuremath{\mathcal{D}}}

\theoremstyle{plain}
\newtheorem{prop}{Proposition}
\newtheorem{teo}{Theorem}
\newtheorem{lema}{Lemma}

\theoremstyle{definition}
\newtheorem{defi}{Definition}

\theoremstyle{remark}
\newtheorem{remark}{Remark}

\newcommand\Osq{\mathbin{\text{\scalebox{.84}{$\square$}}}}

\usepackage{centernot}

\def\ra{\rightarrow}

\def\llr{\longleftrightarrow}
\def\nllr{\centernot\longleftrightarrow}

\providecommand{\keywords}[1]
{
  \small	
  \noindent\textbf{\textit{Keywords---}} #1
}

\title{Percolation with random one-dimensional reinforcements}
\author{A. Nascimento$^{1,*}$, R. Sanchis$^{2,\dagger}$, D. Ungaretti$^{3,\ddag}$}

\begin{document}
\maketitle
\begin{center}
$^*$alanbruno@ufmg.br, $^\dagger$rsanchis@mat.ufmg.br, $^\ddag$daniel@im.ufrj.br\\
    \noindent$^{1,2}${\it Universidade Federal de Minas Gerais}\\
$^3${\it Universidade Federal do Rio de Janeiro}
\end{center}
\begin{abstract}
 We study inhomogeneous Bernoulli bond percolation on the
graph $G \times \ZZ$, where $G$ is a connected quasi-transitive graph.  The inhomogeneity is introduced through a  random region $R$ around the {\it origin axis} $\{0\}\times\ZZ$, where each edge in $R$ is open with probability $q$ and all other edges are open with probability $p$. When the region $R$ is defined by stacking or overlapping boxes with random radii centered along the origin axis, we derive conditions on the moments of the radii, based on the growth properties of $G$, so that for any subcritical $p$ and any $q<1$, the non-percolative phase persists.

\end{abstract}
\textit{Mathematics Subject Classification (2020):} 60K35, 60K37.

\noindent\keywords{Inhomogenous percolation,  random reinforcements.}
\section{Introduction}

Percolation models the spread of a fluid through random medium and it has been object of intensive study since it was introduced in 1957 by Broadbent and Hammersley \cite{hammersley}.  Classical models consider a homogeneous medium representing it as a random graph where edges (or sites) are independently present with probability $p$ or absent otherwise.

The fluid is regarded as the connected component, also called the open cluster, of a fixed point that we call the origin of the graph. With the development of many techniques in the 80's, many questions about inhomogeneous percolation were raised. That is, when some region $R$ of privileged flow is considered in the medium. We now formalize these ideas and give concrete examples of such inhomogeneous models and questions.

We say that a graph $G=(V,E)$ is quasi-transitive if there exists a finite set of sites $V_0\subset V$ such that for every site $w\in V$, there exists $x\in V_0$ and an automorphism $\tau$ of $G$ such that $\tau(y)=x$.   In this work we consider inhomogeneous Bernoulli bond percolation on the cartesian product $G\times \ZZ$ where $G$ is an infinite connected quasi-transitive graph and $\ZZ$ is the set of integers $\ZZ$. The edges of $G\times \ZZ$ are pairs of nearest neighbour's sites and this graph is sometimes called the {\it box product} and denoted $G\Osq \ZZ$. That is,
we consider the graph 
\begin{equation}\label{def:gtimesz}
G\times \ZZ= (V \times \ZZ, E(G \times \ZZ)),
\end{equation}
 where $E(G \times \ZZ)$ is given by the edges which make every layer $V\times \{n\}$ a copy of $G$ and also by the edges connecting the sites corresponding to the same site in $G$ on neighbour layers. More precisely,  if $\sim$  denotes the relation "is connected by an edge to", two sites $(v_1,v_2), (w_1,w_2)\in G\times \ZZ$ are connected by an edge if, and only if, $$v_1=w_1 \text{ and } v_2\sim w_2 \qquad  \text{ or }\qquad v_1\sim w_1 \text{ and } v_2= w_2.$$

We distinguish a vertex in $G$ to be the origin of the graph and we call it $0$ as usual. We also denote by $0$ the origin $(0,0)\in G\times \ZZ$. The set $\{0\}\times \ZZ$ will be called the vertical line or vertical axis along the origin of $G\times \ZZ$. We denote by $B_G^x(r)\subset G$ the set of sites that are up to distance $r$ from $x$, that is
$$
B_G^x(r)= \{v\in G; d_G(x,v)\leq r\}\subset G,
$$
where $d_G$ denotes the graph distance. For simplicity, we will also write $B_G(r)= B_G^0(r)$. We denote the set of integers $[a,b]\cap\ZZ$ within some interval $[a,b]\subset \RR$ simply as $[a,b]\subset \ZZ$.  We denote the box centered at $(x,a)\in G\times\ZZ$ with radius $r$ by 
\begin{equation}\label{eq:Bndef}
B^x(a,r)= B^x_G(r)\times [a-r,a+r]\subset G\times \ZZ
\end{equation}
and we denote for simplicity $B(a,r)= B^0(a,r)$.

 We consider inhomogeneous independent percolation on the graph $G\times\ZZ$.  For a subset $R\subset G\times \ZZ$  let $\P(\text{$e$ is open}) = q$ for edges $e$ in $R$ and $\P(\text{$e$ is open}) = p$ otherwise. Let
us denote the resulting probability measure by $\P^{(R)}_{p,q}$ and its
percolation probability by $\theta^{(R)}(p,q)$, i.e., the probability that there is an infinite open path starting from the origin of $G \times \ZZ$.

The hyphotesis of quasi-transitivity on $G$ is due to some essential properties for our work that such graphs possess. The first one is that quasi-transitive graphs cannot grow too fast. If $\Delta_x$ denotes the degree of $x\in G$ and $\Delta_G=\max_{x\in V_0}\Delta_x$, then, for every $x\in G$,
\begin{equation}\label{eq:expgrowth}
    |B_G^x(n)|\leq \Delta_G^n.
\end{equation}
The second one is that, as it was shown by Antunovi\'{c}, Vaseli\'{c} (see \cite{vaselic2007}), subcritical homogeneous percolation on quasi-transitive graphs always have a sharp threshold, see also Beekenkamp, Hulshof \cite{beekenkamphulshof} for inhomogeneous percolation.  Since $G$ is quasi-transitive, this implies that $G\times\ZZ$ is also quasi-transitive, hence, for any $p< p_c(G\times\ZZ)$, there exists a constant $c=c(p)>0$ such that for every $(x,a)\in G\times\ZZ$,
\begin{equation}\label{eq:expdecay}
    \PP_p((x,a)\llr \partial B^x(a,r))\leq e^{-c(p)r}.
\end{equation}

Our results depend solely on these two properties. While we could present our findings in terms of these properties, we believe that the quasi-transitivity setting is sufficiently general for our purposes.

One important class of problems in inhomogeneous percolation is to consider a parameter $p$ such that $\theta(p)=\theta^{(R)}(p,p)=0$ and a supercritical parameter $q$. We seek to understand the relation between the size and shape of $R$ and how large $q$ must be  to enter the percolating phase of $\P^{(R)}(p,q)$. 

Several positive results are present in the literature when one considers random regions $R$. For instance, Duminil-Copin, Hilário, Kozma, Sidoravicius \cite{hilario} showed that for Brochette Percolation in the square lattice, that is, when $R\subset\ZZ^2$ is  given by the union of vertical lines chosen independently at random, whenever $q>p_c(\ZZ^2)$ we can choose $p<p_c(\ZZ^2)$ such that the origin  percolates. See also \cite{BDS}, \cite{remy},\cite{hoffman}, \cite{KSV} for related problems.

Our work is, though, reminiscent of the following results. In 1994, Madras, Schinazi, and Schonmann \cite{madras} showed that for the case where $G=\ZZ^{d-1}$ and $R=\{0\}\times\ZZ$ is a vertical line, the critical point of the inhomogeneous model remains the same as in the homogeneous case. Their work is actually for the Contact Process, but it translates naturally for percolation. Later in that year, Zhang \cite{zhang} showed that for all $q<1$ we have 

$$\mathbb{P}_{1/2,q}^{\{0\}\times\ZZ }(0\longleftrightarrow\infty)=0,$$
recalling the value of the critical point $p_c(\mathbb{Z}^2)=1/2.$ In other words, phase transition also remains continuous in $\ZZ^2$, which is the case  $d-1=2$ of a problem posed in \cite{madras}. In general,  we can define
\begin{equation}\label{eq:pcqcurve}
q\mapsto p_c(q)= \sup\{p\in[0,1]\,;\, \mathbb{P}_{p,q}^{(R)}(0\longleftrightarrow \infty)=0\}.
\end{equation}

Proposition 1.4 of \cite{madras} shows, in particular, that $p_c(q)$ is a constant curve for $q\in(0,1)$ in $\ZZ^d$ for all $d\geq 2$. In 2020, Szabó and Valesin \cite{valesin} studied this problem for a general graph $G$ and proved that, for any finite subgraph $F\subset G$,  $p_c(q)$ is continuous when $R$ is the cylinder $F\times\ZZ$ in the cartesian product $G\times \ZZ$, which they call a {\it ladder graph}, and they conjecture that the curve is constant. Our results, in particular, imply that this is actually the case when $G$ is quasi-transitive. Also in 2020, Lima and Sanna \cite{bnbl} generalized the result of \cite{valesin} by replacing $F\times \ZZ$ by a region $R$ given by the union of an infinite number of {\it well spaced} cylinders with uniformly bounded radii.

Our objective is to extend beyond the deterministic setting, where the region $R$ is fixed, and investigate models with random thickness in the reinforced one-dimensional region, preventing it from being confined to a deterministic cylinder. We do that in two ways. Firstly, we consider the region $R$ to be the union of boxes centered along the line $\{0\}\times \ZZ$ having radii given by i.i.d. random variables, this is the Overlap Model defined in Subsection \ref{subsection:overlap}. Then we consider the region $R$ given by stacked boxes with i.i.d. radii also centered along the line $\{0\}\times \ZZ$, this is the Stack Model defined in Subsection  \ref{subsection:stack}. In both cases we prove that, under mild conditions on the expectation of the radii, for any $p<p_c(G\times\ZZ)$ and for any $q<1$, the resulting process remains, for almost every environment, in the non-percolating phase.

\section{Definition of the models and statement of the results}

Consider a collection $\{ X_n; n \in \ZZ\}$  of iid  random variables supported
on $\NN=\{1,2,3,\dots\}$ in the probability space $(\Xi,\mathcal{G}, \nu)$, where $\Xi=\NN^{\ZZ}$ and $\mathcal{G}=\sigma(X_n, n\in\ZZ)$, with associated expectation operator denoted by $E$.  One might assume that the random variables $X_n$ are supported on $\mathbb{R}_+$, however, this assumption would not introduce any new phenomena  to our results. A configuration $\Lambda\in\Xi$ will be called an environment.
We start with a classical result of de la Valée Poussin on uniform integrability applied to a single function.

\begin{lema}
\label{lema:improved_moment}
Let $X$ be a random variable supported on $\NN$. If $\E X<\infty$, then there exists a non-decreasing function $f : \RR \to[0,\infty)$ with $f(x) \nearrow \infty$ as $x \to \infty$, such that 
$$
\E[ Xf(X)]<\infty.
$$
Moreover, if $g$ is the inverse function of $x f(x)$.
Then, we have
$$
\lim_{n\ra \infty} \frac{g(n)}{n}=0.
$$
\end{lema}

\begin{proof}

We prove only the last assertion, by contradiction. Suppose that there exists some
$\varepsilon > 0$ and an increasing sequence $(n_k)$ with $g(n_k) \ge \varepsilon n_k$.
Then, applying $x \mapsto x f(x)$ we get
\begin{equation*}
n_k \ge \varepsilon n_k f(\varepsilon n_k), \ \forall k
\quad \implies \quad
   f(\varepsilon n_k)\leq  \frac{1}{\varepsilon} , \ \forall k
\end{equation*}
contradicting that $f(n) \nearrow \infty$.
\end{proof}

From now on, we focus on two possible models of random environments describing one-dimensional reinforcements. 
\vspace{5pt}

\subsection{The Overlap Model}\label{subsection:overlap}
In this subsection we introduce formally the Overlap Model and we state the main theorem. Let $G$ be a quasi-transitive graph and consider the cartesian product $G\times \ZZ$ as defined in \eqref{def:gtimesz}.
    In each site $(0,n)$ of the vertical line $\{0\} \times \ZZ$, we place a box 
    \begin{equation}\label{eq:overlapBndef}
 B_n=B(n,X_n)
    \end{equation}

    with radius $X_n$ and we consider the improved region $R$ to be the union of these boxes. More precisely, we define
        \begin{equation}\label{eq:Roverlap}
            R = \bigcup_{n\in\ZZ} B_n.
        \end{equation}
        We let every edge $e\in R$ be open with probability $q$ and every edge $e\notin R$ open with probability $p$. In Figure \ref{figure:overlapmodel}, we sketch an environment for the Overlap Model for $G=\ZZ$. To let the drawing clear and to not lose proportionality, we choose to draw only boxes with centers $10$ units apart. The reader is invited to imagine the other boxes in between them.
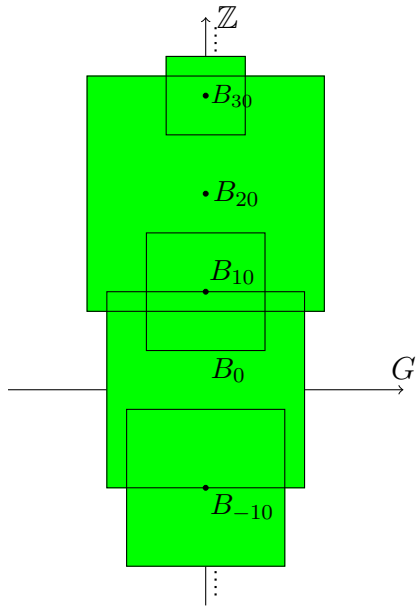
\begin{figure}[ht]
    \centering
    \begin{tikzpicture}[scale=1.3]
        \draw[->] (0,-2.2) --++ (0,6) node[right]{$\ZZ$};
    \draw[->] (-2,0) --++ (4,0) node[above]{$G$};
    \filldraw[black] (0,0) circle (.7pt) ;
    \begin{scope}[transparency group, opacity=0.2]

\filldraw [green] (-0.4,1.4) rectangle (0.4, 0.6);

\filldraw [green] (-1,1) rectangle (1, -1);
\filldraw [green] (-1.2,3.2) rectangle (1.2, 0.8);
\filldraw [green] (-0.8,-0.2) rectangle (0.8, -1.8);
\filldraw [green] (-0.4,3.4) rectangle (0.4, 2.6);
\end{scope}
\begin{scope}[transparency group, opacity=0.3]

\draw [black] (-0.6,1.6) rectangle (0.6, 0.4);
\draw [black] (-1,1) rectangle (1, -1);
\draw [black] (-1.2,3.2) rectangle (1.2, 0.8);
\draw [black] (-0.8,-0.2) rectangle (0.8, -1.8);
\draw [black] (-0.4,3.4) rectangle (0.4, 2.6);
\end{scope}

\node at (0.23,0.2) {\small{$B_0$}};
\node at (0.27,1.2) {\small{$B_{10}$}};
\node at (0.31,2) {\small{$B_{20}$}};
\node at (0.37,-1.2) {\small{$B_{-10}$}};
\node at (0.27,3) {\footnotesize{$B_{30}$}};

    \filldraw[black] (0,1) circle (.7pt) ;
    \filldraw[black] (0,-1) circle (.7pt) ;
    \filldraw[black] (0,3) circle (.7pt) ;

    \filldraw[black] (0,2) circle (.7pt) ;

    {\draw[dotted, thick] (0.1,3.45)--++ (0,0.25);
    \draw[dotted, thick] (0.1,-1.85)--++ (0,-0.25);
    }
    \end{tikzpicture}
    \caption{An environment for the Overlap Model}
    \label{figure:overlapmodel}
\end{figure}

The region $R=R(\Lambda)$ is well defined and we write simply $\P^{(R(\Lambda))}_{p,q}=\P^{\Lambda}_{p,q}$ and $\theta^{(R(\Lambda))}(p,q)=\theta^{\Lambda}(p,q)$. We call 
$\P^{\Lambda}_{p,q}$ the {\it quenched} probability measure associated with the environment $\Lambda$. 

Our main result for this model is the following:

\begin{teo}
\label{teo:overlap}
    Let $G$ be a quasi-transitive graph and consider the Overlap Model defined by iid random variables $\{X_n; n\in\mathbb{Z}\}$ with common distribution $X$ in $G\times\ZZ$. For $0<p < p_c< q < 1$, we have
\begin{equation}
\label{eq:teo_first_moment}
\theta^\Lambda(p,q) = 0\quad \Lambda\text{-a.s. }
\quad \iff \quad
E X < \infty.
\end{equation}

\end{teo}

\begin{remark} Considering environments with radii given by random variables such that $X\geq r$ a.s., it is clear that any cylinder $F\times\ZZ$ is contained in the region $R$ if $r$ is chosen properly. In this setting the Overlap Model dominates the model where the improved region $R$ is a cylinder. In the context of quasi-transitive graphs $G$, this proves the conjecture by Szabó and Valesin \cite{valesin}  that the curve $p_c(q)$ defined in \eqref{eq:pcqcurve} is constant if the improved region is a deterministic cylinder.
\end{remark}

\begin{remark} In the Overlap Model as defined above, each edge in the enhanced region is open with a fixed probabilily $q<1$ and one could suppose that this probability increases  depending on how many boxes the edge belongs to. As it will be clear, our proof could be easily adapted to such case and the result would be the same. This relates to the fact that almost surely,
either every edge belongs to a finite number of boxes or the boxes cover the
whole space, see Proposition \ref{prop:infinite_first_moment}.

\end{remark}

       \subsection{The Stack Model}\label{subsection:stack}
In this subsection we define formally the Stack Model and state the main theorem.

To start the construction of the environment, place the box  $B(0,X_0)$ at the origin of $ G\times\ZZ$ and then stack succssivelly the remaining boxes in both directions along $\{0\}\times\ZZ$. More precisely. Let  $Z(0)=0$ and, for $n\in\mathbb{Z}\setminus\{0\}$,
\begin{equation}\label{eq:centerzn}
Z(n)=\begin{cases}
X_0+2\sum_{i=1}^{n-1} X_i +X_n,\quad\quad\text{ if }n\geq 1\\
    -X_0-2\sum_{i=1}^{-n-1} X_{-i} -X_{n},\text{ if }n\leq -1
\end{cases}
\end{equation}
be the center of the $n$-th box and set \begin{equation}\label{eq:stackBn}
    B_n=B(Z(n),X_n).
\end{equation} 
The improved region $R$ is again given by 
\begin{equation}\label{eq:Rstack}
R= \bigcup_{n\in\ZZ}B_n.
\end{equation}

       \begin{figure}[ht]
            \centering
            \begin{tikzpicture}
                \draw[->] (0,-3) --++ (0,7.3) node[right]{$\ZZ$};
    \draw[->] (-2,-0.8) --++ (4,0) node[above]{$G$};
{   \draw [fill=green, draw=black, fill opacity= 0.2] (-1,-1.8) rectangle (1, 0.2);

\node at (0.5,-0.5) {$B_0$};}
{\draw [fill=green, draw=black, fill opacity= 0.2] (-0.3,0.2) rectangle (0.3, 0.8);
\draw[fill=green, draw=black, fill opacity= 0.2] (-0.4,-1.8) rectangle (0.4,-2.6);}
{\draw [fill=green, draw=black, fill opacity= 0.2] (-0.4,0.8) rectangle (0.4, 1.6);}
{\draw [fill=green, draw=black, fill opacity= 0.2] (-0.8,1.6) rectangle (0.8, 3.2);
%\draw[thick, dotted] (0,2.4)--++(0.8,0);
\node at (0.4, 2.6) {$B_3$};}
{\draw [fill=green, draw=black, fill opacity= 0.2] (-0.25,3.2) rectangle (0.25, 3.7);}
   {\draw[dotted, thick] (0.1,3.8)--++ (0,0.25);
    \draw[dotted, thick] (0.1,-2.7)--++ (0,-0.25);
    }
            \end{tikzpicture}
            \caption{An environment for the Stack Model}
            \label{figure:stackmodel}
        \end{figure}
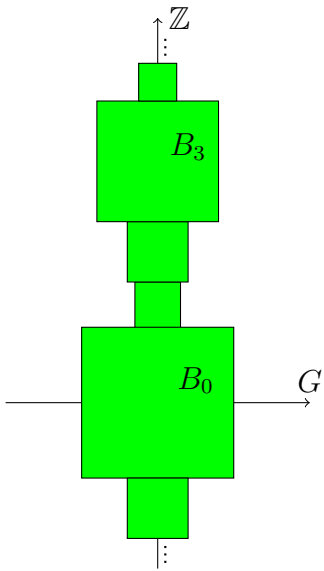
    As in the Overlap Model, the region $R=R(\Lambda)$ is well defined and we let every edge $e\in R$ be open with probability $q$ and every edge $e\notin R$ open with probability $p$ (see Figure \ref{figure:stackmodel}). For the sake of simplicity, we use the same notation to denote the probability measure associated to the Stack Model, that is  $\P^{(R(\Lambda))}_{p,q}=\P^{\Lambda}_{p,q}$ and $\theta^{(R(\Lambda))}(p,q)=\theta^{\Lambda}(p,q)$.

      In the Overlap Model, the arrangement of the center of the boxes  produces more density of boxes along a determined height, causing the region $R$ to be {\it large} when the random variables $X_n$ are fixed. In the Stack Model, instead,  at every fixed height the sample of only one box will determine the thickness of the improved region $R$. Moreover, we emphasize that in the Stack Model $R$ cannot ever be the whole space. Considering connectedness of $R$ preserved, this is the most sparse way that the boxes can be arranged along $\{0\}\times \ZZ$, being in a sense the extreme opposite of the Overlap Model. For that reason, although we see no straightforward coupling so that Overlap Model dominates Stack Model, the Stack Model produces sparser improved regions, at an heuristic level. As an effect, a weaker hypothesis is needed to prove essentially the same result.

\begin{teo}\label{teo:stack}
    Let $G$ be any quasi-transitive graph and consider the Stack Model in  $G\times\ZZ$. If $E\log|B_G(X)|<\infty$, then

    \begin{equation*}
\theta^\Lambda(p,q) = 0\quad \text{for $p\in [0,p_c), \;q\in [0,1)$}
   \end{equation*}
    almost surely on $\Lambda$.
\end{teo}

The condition $E\log|B_G(X)|<\infty$ in Theorem \ref{teo:stack} expresses the exact relationship between the distribution of $X$, the growth of the graph $G$ and the exponential decay \eqref{eq:expdecay} that must be satisfied for the theorem to hold. In fact, as we will see in Section \ref{section:stackperc}, the $\log$ function appears in the expression as a consequence of the exponential decay of $\theta_n(p)$.

By the fact  that quasi-transitive graphs grow at most exponentially fast, see \eqref{eq:expgrowth}, Theorem \ref{teo:stack} guarantees, for any such graph, that there is no percolation in the Stack Model whenever $EX<\infty$. But if the asymptotic behavior of $|B_G(n)|$ is known to be lesser than exponential, we can give an explicit condition on the moment of the radii for the theorem to hold. For example, if $G=\ZZ^{d-1}$, then $G\times\ZZ=\ZZ^d$ and in this case we know that there exists a constant $C_d>0$ such that $|B_G(n)|\leq C_d n^d $, so it is sufficient to choose iid radii with $E\log X<\infty$ for the theorem to hold in this case.

\section{The Overlap Model: Proof of Theorem \ref{teo:overlap}}

Consider the Overlap Model with radii given by independent random variables with common distribution $X$ of finite expectation. Take $f$ and $g$ as in Lemma \ref{lema:improved_moment} and using the  estimate from this lemma, we have
\begin{equation*}
    \sum_{n\geq 1} \nu(X_n \ge g(n)) = \sum_{n\geq 1} \nu(X_n f(X_n) \ge n) = \sum_{n\geq 1} \nu(X f(X) \ge n) < \infty.
\end{equation*}
By the Borel-Cantelli's Lemma we can conclude that
\begin{equation}\label{eq:limsupxn}
\nu(\liminf_n \{X_n \le g(n)\}) =
1.
\end{equation}
 Hence, although the region $R$ is random, all except finitely many boxes in $R$ are almost
surely contained in a
deterministic cone, whose growth we can estimate. It will be useful to find a deterministic $n_0$ such that all boxes with center above $n_0$ are within some deterministic region with positive probability.

To formalize these ideas, we start by defining the deterministic cones properly.

\begin{defi}\label{def:overlapcones}
    Given $n_0\in\NN$, such that $g(n)<n/2$ for all $n\geq n_0$. The {\it upwards} and the {\it downwards cone} are, respectively, the sets
    \begin{equation*}
    W^{+} = \bigcup_{n \ge 2n_0} B_G(g(n)) \times \Bigl[\frac{n}{2}, \infty\Bigr)
    \quad\text{and}\quad
    W^{-} = \bigcup_{n \ge 2n_0} B_G(g(n)) \times \Bigl(-\infty, -\frac {n}{2}\Bigr].
\end{equation*}
\end{defi}
Recall the notation $B_n=B(n,X_n)$. See \eqref{eq:Bndef} and \eqref{eq:overlapBndef}.
\begin{defi}\label{def:goodenvironments}
     Given $n_0\in\NN$ and $W^+$, $W^-$ as in Definition \ref{def:overlapcones}, we say that $\Lambda$ is a {\it good} environment if it satisfies
    \begin{enumerate}
        \item The boxes $B_n$ with $n\ge 2n_0$ are contained in $W^+$:\quad
        $\bigcup\limits_{n\geq 2n_0}B_n\subseteq W^+$;
        
        \item The boxes $B_{-n}$ with $n\ge 2n_0$ are contained in $W^-$:\quad
        $\bigcup\limits_{n\geq 2n_0}B_{-n}\subseteq W^-$;
       
        \item The boxes $B_n$ with $n\in(-2n_0,2n_0)$ have radius $X_n\leq g(2n_0)$.
       
    \end{enumerate}
    We denote by $\mathcal{A}=\mathcal{A}(n_0)$ the set of all good environments.
\end{defi}

We notice that in a good environment given by Definition~\ref{def:goodenvironments} the random region $R$ is covered by the deterministic region $\bigl(B_G(g(2n_0)) \times [-n_0,n_0]\bigr) \cup W^+ \cup W^-$. Cones $W^+$ and $W^-$ are illustrated in Figure~\ref{fig:cones} below, where they are decomposed into layers. 

\begin{lema}\label{lema: environment}
There exists $n_0$ such that the set of good environments $\mathcal{A}$ has $\nu(\mathcal{A})>0.$
\end{lema}

\begin{proof}
First notice that the finite intersection of independent events 
$$
\bigcap_{n=-2n_0+1}^{2n_0-1}\{X_n\leq g(2n_0)\}
$$
has positive probability for all $n_0$ such that $g(2n_0)\geq m$, where $m$ is some constant depending on the distribution of $X$.
By Lemma \ref{lema:improved_moment} we can choose $n_0$ sufficiently large so that $g(n)\leq n/2$ for all $n\geq n_0$. Fix $n_0$ with these properties. In order to show that 
$$
\nu\Bigl(\bigcup_{n\geq 2n_0}B_n\subseteq W^+\Bigr)>0
$$
it is sufficient to show that with positive probability $B(n,X_n)\subseteq B(n,g(n))$ for all $n\geq 2n_0$. It follows from \eqref{eq:limsupxn} that 
$$
\nu\Bigl(\bigcap_{n\geq 2n_0}\{X_{n}\leq g(n)\}\Bigr)>0.
$$

By the same reasoning we also show that 
$$
\nu\Bigl(\bigcap_{n\geq 2n_0}\{X_{-n}\leq g(n)\}\Bigr)>0.
$$
 We conclude that $\nu(\mathcal{A})>0$, being the intersection of three independent events of positive probability.
\end{proof}

Proceeding to the proof of the theorem, first we show that one of the implications of the statement is trivial. 
\begin{prop}
\label{prop:infinite_first_moment}
$E X = \infty$ if, and only if, $R=G \times \ZZ$ almost surely.
\end{prop}

\begin{proof}
We show that any fixed vertex $v$ will be covered by some ball. By translation
invariance in the $\ZZ$ direction, we can assume $v = (w, 0)$ and denote 
$r = d_G(w,0)$. Notice that a ball centered at $(0,n)$ with $n > r$
will cover $v$ if its radius satisfies $X_n \ge n$. Since
\begin{equation*}
    \sum_{n \ge r} \nu(X_n \ge n) = \infty
    \quad \text{is equivalent to}\quad 
    E X = \infty,
\end{equation*}
and the events $\{X_n \ge n\}$ are independent, the result follows from
Borel-Cantelli's lemma.
\end{proof}
This proves the first implication since when $R=G\times \ZZ$ we are in the supercritical phase of homogeneous percolation whenever $q>p_c(G\times\ZZ)$.

\begin{proof}[Proof of Theorem \ref{teo:overlap}.]
  We decompose $W^{+}$ and $W^-$ into layers $(L^{+}_n; n \ge n_0)$. Define
\begin{equation}
    \label{eq:layers_W}
    L^{+}_{n} := B_G(g(2n)) \times \{n\}.
\end{equation}
 Notice that $W^+= \bigcup_{n\geq n_0} L_n^+$. In fact, let $n\geq n_0$, the radius of $L_n^+$ at $G\times\{n\}$ is $g(2n)$, which is also the radius of $W^+$ at this height, because the cylinder in $W^+$ with the largest radius that has non empty intersection with $G\times\{n\}$ is  $B_G(g(2n))\times[n,\infty)$.

It follows by quasi-transitivity of $G$ that $B_G(n)$ grows at most exponentially fast with $n$. As a consequence,  the number of sites in each layer is
\begin{equation}
    \label{eq:layers_W_sites}
    |L^{+}_n| = | B_G(g(2n))| \le e^{c_G g(2n)},
\end{equation}
where $c_G>0$ is some constant depending only on $G$.
We define analogously the quantities $L^-_{n}$ for $W^-$.

Let us estimate $\P_p(W^+ \leftrightarrow W^-)$.  We have
\begin{equation*}
\P_p(W^+ \leftrightarrow W^-)
    \le \sum_{n,m \ge n_0/2} \P_p(L_n^+ \leftrightarrow L_m^-).
\end{equation*}
Notice that any $w^+ \in L_n^+$ and $w^- \in L_m^-$ have vertical distance
$m+n$. Since $G$ is quasi-transitive, let $c=c(p)$ be as in $\eqref{eq:expdecay}$.

Summing over all possible pairs, we have
\begin{equation*}
\P_p(L_n^+ \leftrightarrow L_m^-)
    \le |L_n^+||L_m^-|e^{- c (m+n)}\le e^{c_G (g(2n) + g(2m))} e^{- c(m+n)}.
\end{equation*}
Summing for $n, m \ge n_0/2$ we obtain
\begin{equation}\label{entropy}
\P_p(W^+ \leftrightarrow W^-)
    \le  \sum_{n,m \ge n_0/2} e^{c_G (g(2n) + g(2m))} e^{- c(m+n)}
    = 
    \Bigl(\sum_{n \ge n_0/2} e^{n (- c + c_G 2\frac{g(2n)}{2n}) }\Bigr)^2.
\end{equation}

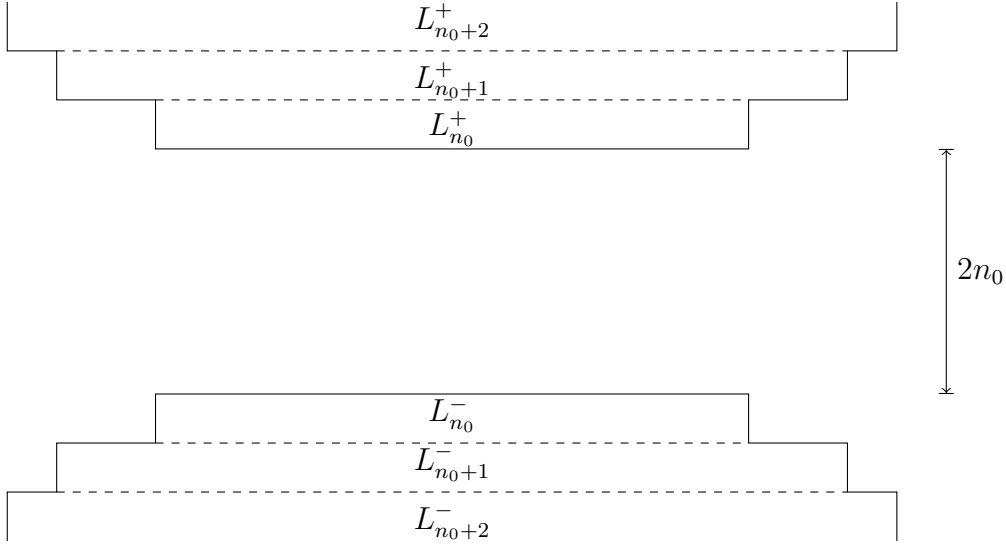
\begin{figure}[ht]
\centering
\begin{tikzpicture}[scale=0.65]
    \draw[|<->|] (10,0) -- ++(0,-5) node[midway, right] {$2n_0$};
\draw (0,0) -- ++(6,0) -- ++(0,1) -- ++(2,0) -- ++ (0,1) -- ++(1,0) --++(0,1);
\draw (0,0) -- ++(-6,0) -- ++(0,1) -- ++(-2,0) -- ++ (0,1) -- ++(-1,0) --++(0,1);
%layers
\draw[dashed] (-6,1) -- (6,1) node[midway, below] {$L_{n_0}^{+}$};
\draw[dashed] (-8,2) -- (8,2) node[midway, below] {$L_{n_0+1}^{+}$};
\node[above] at (0,2) {$L_{n_0+2}^{+}$};

    \begin{scope}[shift={(0,-5)}]
\draw (0,0) -- ++(6,0) -- ++( 0,-1) -- ++( 2,0) -- ++ (0,-1) -- ++( 1,0) --++(0,-1);
\draw (0,0) -- ++(-6,0) -- ++(0,-1) -- ++(-2,0) -- ++ (0,-1) -- ++(-1,0)
    --++(0,-1);
%layers
\draw[dashed] (-6,-1) -- (6,-1) node[midway, above] {$L_{n_0}^{-}$};
\draw[dashed] (-8,-2) -- (8,-2) node[midway, above] {$L_{n_0+1}^{-}$};
\node[below] at (0,-2) {$L_{n_0+2}^{-}$};
\end{scope}
\end{tikzpicture}
\caption{Cones and their layers}
\label{fig:cones}
\end{figure}
Recalling that by Lemma~\ref{lema:improved_moment} we have $g(n)/n\ra0$,  we conclude that the sum of the series above is
convergent and we can actually make it as close to zero as we want by increasing $n_0$.

Now we proceed to transport this result to the measure $\P_{p,q}^\Lambda$. Fix $n_0$ large enough so that
$$
\P_p(W^+ \llr W^-)\leq \frac12.
$$
 
 Let $B= B_G(g(2n_0)) \times [-n_0, n_0]$. Recall that for good environments $\Lambda\in \mathcal{A}$, the region $R$ is a subset of $B\cup W^+\cup W^-$. Let $D=\{W^+ \nllr W^-\}$ and 
$
F=\{\text{every edge inside $B$ is closed}\}.
$ 
Notice that,
\begin{align}
\P_{p,q}^\Lambda(D^c\cap F)&=\P_{p,q}^\Lambda(\{W^+\xleftrightarrow{B^c} W^-\}\cap F)\nonumber\\
&= \P_{p,q}^\Lambda(W^+\xleftrightarrow{B^c} W^-)\P_{p,q}^\Lambda(F)\nonumber\\
& =\P_{p}(W^+\xleftrightarrow{B^c} W^-)\P_{p,q}^\Lambda(F)\nonumber\\
& \leq \P_{p}(D^c)\P_{p,q}^\Lambda(F).\nonumber
\end{align}
As a consequence, we have 
\begin{equation}\label{eq:Ddecreasing}
\P_{p,q}^\Lambda(D\cap F)=\P_{p,q}^\Lambda(F)- \P_{p,q}^\Lambda(D^c\cap F)\geq \P_{p,q}^\Lambda(F)\P_{p}(D).
\end{equation}
Let
\begin{equation}
R^+= \bigcup_{n\geq n_0} B(n,X_n)
\quad \text{and}\quad
R^-= \bigcup_{n\geq n_0} B(-n,X_{-n}).
\end{equation}

Since $\Lambda\in\mathcal{A} $, we have  that $R^\pm\subset W^\pm$. Using this fact and \eqref{eq:Ddecreasing},  we get
\begin{align}\label{lifteq}
\PP^{\Lambda}_{p,q}(R^+\nleftrightarrow R^-)&\geq \PP^{\Lambda}_{p,q}(D) \ge \PP^{\Lambda}_{p,q}(D \cap F)\\
&\geq \PP^{\Lambda}_{p,q}(F) \PP_{p}(D)\geq \frac{1}{2}(1-q)^{c(G,n_0)},\nonumber
\end{align}
where $c(G,n_0)$ is is the number of edges in $B_G(g(2n_0)) \times [-n_0, n_0]$.

Now, in order to use ergodic properties, we define the  {\it annealed} law
$$
\P_{p,q}^\nu(\cdot)= \int_\Xi\P_{p,q}^\L(\cdot)d\nu(\L).
$$
 The main idea is that when we show that an event has propability $1$ in the annealed law, then it also has probability $1$ for $\nu$-almost all environment $\Lambda$.

We have,
$$
\P_{p,q}^\nu(D)=\int_\Xi\P_{p,q}^\L(D)d\nu(\L)\geq \int_{\mathcal{A}}\P_{p,q}^\L(D)d\nu(\L)\geq \frac12 (1-q)^{c(G,n_0)}\nu(\mathcal{A})>0.
$$

Now, let $\tau$ be the vertical translation and consider the event where there exists $n$ such that  $\tau^n(D)$ occurs. Clearly this event is invariant by $\tau$. By ergodicity of the annealed measure with respect to vertical translations, we conclude that there are infinitely many vertical disconnections almost surely, and this implies that $\theta^\Lambda(p,q)=0$ $\nu$-almost surely. Indeed, in order for the cluster of the origin to be infinite whenever $p<p_c(G\times\ZZ)$, it has to have infinite intersection with $R$.  
\end{proof}

\section{The Stack Model: Proof of Theorem \ref{teo:stack}}\label{section:stackperc}

In this section we prove Theorem \ref{teo:stack}. We actually prove it in a slightly more general setting. We choose a function  $\varphi\geq 0$ such that $\varphi\nearrow\infty$, $E\varphi(X)<\infty$ and we investigate what additional properties it has to satisfy for the result to hold.

Let $\varphi\geq 0$ be any increasing   function and suppose $E\varphi(X)<\infty$. 
As $X$ is unbounded, the set $\mathcal{L}=\{m\in\NN \,|\,\nu(m\leq X\leq 2m)>0\}$ is infinite. 
\begin{lema}\label{lema:stackAk}
    Let  $L=\min\{l\in\NN\,|\,  \nu(\varphi(X)\leq l)>0\}$,  $l_0\in\mathcal{L}$  and consider the event 
$$
A_k= \{l_0\leq X_{k}\leq 2l_0\}\cap\bigcap_{j\geq 1}\{X_{k+j}\leq \varphi^{-1}(j+L)\}\cap\bigcap_{j\geq 1}\{X_{k-j}\leq \varphi^{-1}(j+L)\}.
$$
Then, for every $k\in\ZZ$,
$
\nu(A_k)>0. 
$
Moreover, $\nu(A_k)$ is constant as a function of $k$.
\end{lema}
\begin{proof}
As the sequence of random variables is iid, we have
\begin{align*}
    \nu(A_k)=\nu(l_0\leq X\leq2l_0)\biggl(\prod_{j\geq 1} \nu(\varphi(X)\leq j+L)\biggr)^2
\end{align*}
which is positive since $E\varphi(X)<\infty$. Furthermore, the identity also shows that $\nu(A_k)$ does not depend on $k$. 
\end{proof}

Recall from \eqref{eq:stackBn} that $B_n=B(Z(n),X_n)$. For $k\in\ZZ$, it will also be useful to define the subregions
\begin{equation}\label{eq:stackR+}
R^+_k=\bigcup_{n> k} B_n
\end{equation}
and
\begin{equation}\label{eq:stackR-}
R^-_{k}=\bigcup_{n< k} B_n.
\end{equation}
Notice that for any $k\in\ZZ$, $R=R^+_k\cup B_k\cup R^-_k$.
Now we proceed to define analogous structures as done for the Overlap Model in Definition \ref{def:overlapcones}.

\begin{defi}\label{def:stackcones} 
For $k\in\ZZ$, we define the  {\it upwards cone} $W^+_k$ and the {\it downwards cone} $W^-_k$ as the sets

$$
W^+_k= \bigcup_{n\geq 0} B_G(\varphi^{-1}(n+L+1))\times [Z(k)+l_0+n,\infty)
$$
and
$$
W^-_k= \bigcup_{n\geq 0} B_G(\varphi^{-1}(n+L+1))\times (-\infty,Z(k)-l_0-n].
$$
\end{defi}

\begin{prop}\label{prop:coneinequality}
Let $c=c(p)$ be as in \eqref{eq:expdecay}. There exists $l_0$ large enough such that for an environment $\Lambda\in A_k$, if 
\begin{equation}\label{eq:limsup2}
      \sum_{n\geq 1}|B_G(\varphi^{-1}(n+L))|e^{-cn}<\infty
   \end{equation}
then, there exists a constant $c(G,l_0)$ such that 
\begin{equation}\label{eq:propconeinequality}
    \P_{p,q}^\L(R^+_{k}\nleftrightarrow R^-_{k})\geq\frac12 (1-q)^{c(G,l_0)}>0.
\end{equation}
\end{prop}

\begin{proof}
For any environment $\Lambda \in A_k$ we have $R^+_{k}\subset W^+_{k}$ and $R^-_{k}\subset W^-_{k}$ for all $k\geq 1$, so with the same flavor of \eqref{entropy}, for $\Lambda\in A_k$, we have
\begin{align}\label{connection}
\P_p(R^+_{k}\longleftrightarrow R^-_{k})&\leq \P_p(W^+_{k}\longleftrightarrow W^-_{k})\nonumber\\
&\leq  e^{-2cl_0}\sum_{m,n\geq 1}{|B_G(\varphi^{-1}(n+L))||B_G(\varphi^{-1}(m+L))|}e^{-c(m+n)}\nonumber\\
&=e^{-2cl_0}\biggl(\sum_{n\geq 1}{|B_G(\varphi^{-1}(n+L))|}e^{-cn}\biggr)^2.
\end{align} 

Now, by hypothesis  $\sum_{n\geq 1}{|B_G(\varphi^{-1}(n+L))|}e^{-cn}$ is convergent, so we can make the probability of connection as low as we want uniformly on $k$, by choosing $l_0$ large enough. In this stage, we follow the same procedure of \eqref{lifteq} in Theorem 1 with $\L\in A_k$ and
 $$F=\{\text{every edge of } B_k \text{ is closed}\}$$ to obtain
\begin{align}\label{disconnection}
\P_{p,q}^\L(R^+_{k}\nllr R^-_{k})\geq\frac12 (1-q)^{c(G,l_0)}>0
\end{align}
where $c(G,l_0)\geq 1$ is the number of edges in $B(0,2l_0)$.  
\end{proof}

The next proposition shows that the lower bound \eqref{eq:propconeinequality} in Proposition \ref{prop:coneinequality} holds under the hypothesis of Theorem \ref{teo:stack}.
\begin{prop}\label{prop:stackhyp}
    Let $X$ be a random variable such that $E\log |B_G(X)|<\infty$. Then, for all $p<p_c(G\times \ZZ)$,  there exists an increasing function $\varphi=\varphi_p$ such that $E\varphi(X)<\infty$ and 

    \begin{equation*}
      \sum_{n\geq 1}|B_G(\varphi^{-1}(n+L))|e^{-cn}<\infty
   \end{equation*}
\end{prop}

\begin{proof}
Let $\alpha=c(p)/2$  and define $g(n)= |B_G(n)|$, $\varphi(n)=\alpha^{-1}\log (g(n)).$ Then $E\varphi(X)<\infty$, $\varphi^{-1}(n)= g^{-1}(e^{\alpha n})$ and
\begin{equation*}
\sum_{n\geq 1}|B_G(\varphi^{-1}(n+L))|e^{-cn}=e^{\alpha L}\sum_{n\geq 1}e^{-(c-\alpha)n}<\infty.
\qedhere
\end{equation*}
\end{proof}

\vspace{8pt}
\begin{remark}
   There is a crucial difference on the dynamics of the models with respect to the shift transformation regarding the existence or not of the first moment of $X$, which is a key point for both proofs. In Theorem \ref{teo:stack} we cannot always rely on ergodicity in the same way as in Theorem \ref{teo:overlap}  because  the renewal process with interarrival distribution given by the Law of $X$ has no stationary measure if $EX=\infty$.
\end{remark}

\begin{proof}[Proof of Theorem~\ref{teo:stack}.]
We start by defining a sequence of explorations on the boundaries of the downward cones in order to verify the existence of infinitely many vertical blockades for the origin's cluster almost surely. Loosely speaking, we explore the cluster of all sites in the  downwards cone to see if any of them intersects the corresponding upwards cone. That is, given $\omega\in \Omega$, we choose an enumeration $(e_k)_k$ of the edges of $G\times\ZZ$ and, for every $l\in\NN$, define $\mathcal{C}_0=\mathcal{C}_0(l)= W^-_{l}$, $\D_0(l)=\emptyset$. Given a set of sites $S\in G\times \ZZ$, the {\it edge boundary} of $S$ is the set $\partial S=\{e=vw\, ;\, v\in S, w\notin  S\}. $
\begin{enumerate}
    \item Let $m_1$ be the smallest index $k$ such that $e_{k}\in \partial \mathcal{C}_0$. We set $\mathcal{D}_1=\{e_{m_1}\}$ and if the edge $e_{m_1}=v_{m_1}w_{m_1}$ is open, with $w_{m_1}\notin \cC_0$, then we set $\mathcal{C}_1=\C_0\cup\{w_{m_1}\}$, otherwise we set $\mathcal{C}_1=\C_0$.

     \item For $n\geq 2$, we let $m_n$ be the smallest index $k$ such that $e_{k}\in \partial \C_{n-1}\setminus\cD_{n-1}$. We set
     $
     \D_n=
         \D_{n-1}\cup \{e_{m_n}\}
     $ and
     $$
     \C_n=\begin{cases}
         \C_{n-1}\cup\{w_{m_n}\}, \text{if } e_{m_n}=v_{m_n}w_{m_n} \text{ is open}, \,w_{m_n}\notin \cC_{n-1}\\
         \C_{n-1}, \text{ otherwise}.
     \end{cases}
     $$

     \item  If $\C_n\cap W^+_{l}\neq \emptyset$ the exploration stops and we say that {\it the exploration failed} and we define $\C_i=\C_n$ for $i\geq n$. If the exploration does not stop, then we say that {\it the exploration succeeded}.
\end{enumerate}

Notice that by construction the exploration does not stop only in configurations where there is no crossing $\partial W_{l}^-\leftrightarrow \partial W_{l}^+$. The exploration of $\partial W^-_l$ just defined is denoted by $\mathcal{E}_l$.

Since $(X_n)_n$ is an iid sequence, it is ergodic with respect to the shift transformation, thus
$$
\nu(A_k \;{ i.o.})=1.
$$
In fact, by Birkhoff's Ergodic Theorem,
$$
\lim_{n\ra\infty}\frac1n \sum_{k=0}^{n-1}\I_{A_k}=\nu(A_0)>0
$$
$\nu$-almost surely.

In other words, the set $\mathcal{N}(\Lambda)=\{l\in \NN\,|\,A_l \text{ occurs}\}$ is infinite almost surely. Let $n_1=n_1(\Lambda)=\min\mathcal{N}$ and consider the exploration $\mathcal{E}_{n_1}$.
 Suppose the exploration failed and for $m\geq 1$ let $\mathcal{C}(W_{m}^-)=\bigcup_{i=1}^\infty \C_i(m)$.  In this case, as the exploited region is finite, we can find a translation of $W^-$ that contains it. That is, we can choose the next index $n_2=n_2(\omega,\Lambda)\in\mathcal{N}$ such that 
$$
W^-_{n_1}\cup\mathcal{C}(W_{n_1}^-)\subseteq W^-_{n_2},
$$
and perform independently the same exploration on $\partial W_{n_2}^-$.

More generally,  the exploration $\mathcal{E}_{n_k}$ of the cluster of $W^-_{n_k}$ is well defined for every $k$ and explored only a finite number of edges, where $n_k=n_k(\omega,\Lambda)$ is the smallest index such that $$W^-_{n_{k-1}}\cup\mathcal{C}(W_{n_{k-1}}^-)\subseteq W^-_{n_k},$$ and it is independent of the preceeding explorations.

Define $\mathcal{F}_k=\sigma(\mathcal{E}_{n_1},\dots,\mathcal{E}_{n_k})$, the smallest $\sigma$-álgebra that contains all the information of the first $k$ explorations and let $T^+$ denote the index of the first exploration to succeed. That is
\begin{equation}\label{eq:T+}
T^+(\Lambda,\omega)= \min\{k\geq 1; \;\mathcal{E}_{n_k} \text{ succeeded}\}.
\end{equation}
Notice that, as $\mathcal{E}_m$ is independent of $\mathcal{F}_{m-1}$, we have
\begin{align}\label{eq:conditionalexp}
   \P_{p,q}^\Lambda\left(T^+>m\right)&=\E\left[\E\left[\I_{\bigcap_{k=1}^{m}\{\mathcal{E}_{n_k} \text{ failed}\}} \bigg| \mathcal{F}_{m-1}\right]\right]\\
    &=\E\left[\I_{\bigcap_{k=1}^{m-1}\{\mathcal{E}_{n_k} \text{ failed}\}}\E\left[\I_{\{\mathcal{E}_{n_m} \text{ failed}\}} \bigg| \mathcal{F}_{m-1}\right]\right]\nonumber\\
    &=\P_{p,q}^\Lambda(\mathcal{E}\text{ failed})\P_{p,q}^\Lambda\left(T^+>m-1\right).\nonumber
\end{align}
Proceeding by induction, Proposition \ref{prop:coneinequality} and Proposition \ref{prop:stackhyp}, we have that for every $m\ge 1$,
$$
\P_{p,q}^\Lambda\left(T^+>m\right)=\prod_{k=1}^m\P_{p,q}^\Lambda(\mathcal{E}_{n_k} \text{ failed})\leq \left(\frac12 (1-q)^{c(G,l_0)}\right)^m\leq \frac{1}{2^m},
$$
hence
$$
\P_{p,q}^\Lambda\left(T^+=\infty\right)=\P_{p,q}^\Lambda\left(\bigcap_{m=1}^\infty \{T^+>m\}\right)=\lim_{m\rightarrow\infty}\P_{p,q}^\Lambda\left(T^+>m\right)\leq\lim_{m\rightarrow\infty}2^{-m}=0.
$$

Until now, we are investigating whether the origin's cluster is infinite by exploring the boundary of cones such that the initial vertex of exploration is on the positive side of the vertical axis. Analogously, one can define the mirrored explorations $\mathcal{E}_{-k}$ of cones with initial vertex of exploration on the negative side of the vertical axis, we call these {\it downward explorations}. Notice that the downward explorations are exploring the boundaries of the upward cones. Thus, we can also define the first index $T^-$ such that a downward exploration succeeds. That is, 
\begin{equation}\label{eq:T-}
    T^-(\Lambda,\omega)= \min\{k\geq 1; \;\mathcal{E}_{n_{-k}} \text{ succeeded}\}.
\end{equation}

To prove the theorem it is sufficient to show that some exploration will succeed almost surely on both upper and lower semispaces. That is, it is sufficient to show that $$\P_{p,q}^\Lambda(T^+=\infty)=\P_{p,q}^\Lambda(T^-=\infty)=0.$$

By symmetry we also have that $\P_{p,q}^\Lambda\left(T^-=\infty\right)=0$ and the proof is, thus, completed.
\end{proof}

\section{Open Questions}

We conclude this article mentioning some questions that we think would be interesting to investigate further.

\begin{enumerate}
    
\item The Stack Model results in a connected region. Alternatively, one could consider boxes with random spacing between them and set the parameter $q=1$. Our methods might be applicable in this scenario, and it would be interesting to establish conditions on the radii and spacings necessary to achieve the percolating phase.

\item The models considered here are intrinsically one-dimensional, with the reinforced region centered along a fixed line. Now, consider a graph $G$ where the simple symmetric random walk is transient. An interesting question is what occurs if the path of the random walk is reinforced or thickened by random boxes around it.

\item A natural question we are unable to tackle is what happens when $p$ is exactly the critical point of $\ZZ^2$. The simplest problem is to prove that the Zhang's result can be extended to our setting. Perhaps the current understanding of near-critical percolation on $\ZZ^2$ could be useful in addressing this issue.

\item The reinforced region in the Stack Model is never the entire graph. It would be nice to check if our result is sharp for $\Z^d$.  Does percolation occur for large $q$ whenever $\E(\log(X))=\infty$?  

\end{enumerate}

\noindent
{\bf Acknowledgments}: The authors thank Augusto Teixeira for suggesting the problem. A.N thanks CAPES, R.S was partially supported by CNPq, CAPES and by FAPEMIG, grants APQ-00868-21 and RED-00133-21, D.U. thanks FAPERJ, grant E-26/210.571/2023.

\bibliographystyle{abbrv}
\bibliography{main}

\end{document}